\documentclass[draft]{amsart}

\usepackage{color}

\usepackage{pgf,tikz}
\usepackage{amssymb}
\usepackage{enumerate}
\usepackage{stmaryrd}

\newcommand{\wTilde}{\widetilde}

\numberwithin{equation}{section}
\theoremstyle{plain}
\newtheorem{theorem}{Theorem}[section]
\newtheorem{lemma}[theorem]{Lemma}
\newtheorem{corollary}[theorem]{Corollary}
\newtheorem{proposition}[theorem]{Proposition}

\begin{document}

\makeatletter
\def\imod#1{\allowbreak\mkern10mu({\operator@font mod}\,\,#1)}
\makeatother

\author{Alexander Berkovich}
   \address{Department of Mathematics, University of Florida, 358 Little Hall, Gainesville FL 32611, USA}
   \email{alexb@ufl.edu}

\author{Ali Kemal Uncu}
   \address{Department of Mathematics, University of Florida, 358 Little Hall, Gainesville FL 32611, USA}
   \email{akuncu@ufl.edu}

\title[\scalebox{.9}{On partitions with fixed number of even-indexed and odd-indexed odd parts}]{On partitions with fixed number of even-indexed and odd-indexed odd parts}
     
\begin{abstract} 
This article is an extensive study of partitions with fixed number of odd and even-indexed odd parts. We use these partitions to generalize recent results of C. Savage and A. Sills. Moreover, we derive explicit formulas for generating functions for partitions with bounds on the largest part, the number of parts and with a fixed value of BG-rank or with a fixed value of alternating sum of parts. We extend the work of C. Boulet, and as a result, obtain a four-variable generalization of Gaussian binomial coefficients. In addition, we provide combinatorial interpretation of the Berkovich--Warnaar identity for Rogers--Szeg\H{o} polynomials.
\end{abstract}   
   


\date{\today}
   
\maketitle

\section{Introduction and Notation}\label{section1}
%%%%%%%%%%%%%%%%%%%%%%%%%%%%%%%%%%%%%%%%%%%%%%%%%%%%%%%%%%%%%%%%%%%%%%%%%%%%%%%%%%%%%%%%%%%%%%%%%%%%%

A partition $\pi$ is a non-increasing finite sequence $\pi = (\lambda_1,\lambda_2,\dots)$ of positive integers. The elements $\lambda_i$ that appear in the sequence $\pi$ are called parts of $\pi$. For positive integers $i$, we call $\lambda_{2i-1}$ odd-indexed parts, and $\lambda_{2i}$ even indexed parts of $\pi$. We say $\pi$ is a partition of $n$, if the sum of all parts of $\pi$ is equal to $n$. Conventionally the empty sequence is considered as the unique partition of zero. We will abide by this convention. Partitions can be represented graphically in multiple ways. For consistency, we are going to focus on representing partitions using Ferrers diagrams. The $2$-residue Ferrers diagram of partition $\pi$ is given by taking the ordinary Ferrers diagram drawn with boxes instead of dots and filling these boxes using alternating $0$'s and $1$'s starting from $0$ on odd-indexed parts and $1$ on even -indexed parts. We can exemplify 2-residue diagrams with $\pi=(12,10,7,5,2)$ in Table~\ref{table:TBL1}.

\begin{center}
\begin{table}[htb]\label{table:TBL1}\caption{2-residue Ferrers diagram of the partition $\pi=(12,10,7,5,2)$}
\begin{tikzpicture}[line cap=round,line join=round,x=0.5cm,y=0.5cm]
\clip(0.5,0.5) rectangle (14.5,6.5);
\draw [line width=.25pt] (2,1)-- (4,1);
\draw [line width=.25pt] (2,2)-- (7,2);
\draw [line width=.25pt] (2,3)-- (9,3);
\draw [line width=.25pt] (2,4)-- (12,4);
\draw [line width=.25pt] (14,5)-- (2,5);
\draw [line width=.25pt] (2,6)-- (14,6);
\draw [line width=.25pt] (14,6)-- (14,5);
\draw [line width=.25pt] (13,6)-- (13,5);
\draw [line width=.25pt] (12,6)-- (12,4);
\draw [line width=.25pt] (11,6)-- (11,4);
\draw [line width=.25pt] (10,6)-- (10,4);
\draw [line width=.25pt] (9,6)-- (9,3);
\draw [line width=.25pt] (8,6)-- (8,3);
\draw [line width=.25pt] (7,2)-- (7,6);
\draw [line width=.25pt] (6,6)-- (6,2);
\draw [line width=.25pt] (5,6)-- (5,2);
\draw [line width=.25pt] (4,1)-- (4,6);
\draw [line width=.25pt] (3,6)-- (3,1);
\draw [line width=.25pt] (2,6)-- (2,1);
%\draw (1.5,5.5) node[anchor=center] {$\lambda_1$};
\draw (2.5,5.5) node[anchor=center] {0};
\draw (3.5,5.5) node[anchor=center] {1};
\draw (4.5,5.5) node[anchor=center] {0};
\draw (5.5,5.5) node[anchor=center] {1};
\draw (6.5,5.5) node[anchor=center] {0};
\draw (7.5,5.5) node[anchor=center] {1};
\draw (8.5,5.5) node[anchor=center] {0};
\draw (9.5,5.5) node[anchor=center] {1};
\draw (10.5,5.5) node[anchor=center] {0};
\draw (11.5,5.5) node[anchor=center] {1};
\draw (12.5,5.5) node[anchor=center] {0};
\draw (13.5,5.5) node[anchor=center] {1};
%\draw (1.5,4.5) node[anchor=center] {$\lambda_2$};
\draw (2.5,4.5) node[anchor=center] {1};
\draw (3.5,4.5) node[anchor=center] {0};
\draw (4.5,4.5) node[anchor=center] {1};
\draw (5.5,4.5) node[anchor=center] {0};
\draw (6.5,4.5) node[anchor=center] {1};
\draw (7.5,4.5) node[anchor=center] {0};
\draw (8.5,4.5) node[anchor=center] {1};
\draw (9.5,4.5) node[anchor=center] {0};
\draw (10.5,4.5) node[anchor=center] {1};
\draw (11.5,4.5) node[anchor=center] {0};
%\draw (1.5,3.5) node[anchor=center] {$\lambda_3$};
\draw (2.5,3.5) node[anchor=center] {0};
\draw (3.5,3.5) node[anchor=center] {1};
\draw (4.5,3.5) node[anchor=center] {0};
\draw (5.5,3.5) node[anchor=center] {1};
\draw (6.5,3.5) node[anchor=center] {0};
\draw (7.5,3.5) node[anchor=center] {1};
\draw (8.5,3.5) node[anchor=center] {0};
%\draw (1.5,2.5) node[anchor=center] {$\lambda_4$};
\draw (2.5,2.5) node[anchor=center] {1};
\draw (3.5,2.5) node[anchor=center] {0};
\draw (4.5,2.5) node[anchor=center] {1};
\draw (5.5,2.5) node[anchor=center] {0};
\draw (6.5,2.5) node[anchor=center] {1};
%\draw (1.5,1.5) node[anchor=center] {$\lambda_5$};
\draw (2.5,1.5) node[anchor=center] {0};
\draw (3.5,1.5) node[anchor=center] {1};
\end{tikzpicture}
\end{table}
\end{center}

In this paper, we consider partitions with fixed number of odd and even-indexed odd parts. We start our discussion by considering partitions into distinct parts. In this way we are led to Theorem~\ref{BasicCorollaryComb}.

\begin{theorem}\label{BasicCorollaryComb} For non-negative integers $i$, $j$, and $n$
\[p(i,j,n)=p'(i,j,n),\]
where $p(i,j,n)$ is the number of partitions of $n$ into distinct parts with $i$ odd-indexed odd parts and $j$ even-indexed odd parts and $p'(i,j,n)$ is the number of partitions of $n$ into distinct parts with $i$ parts that are congruent to 1 modulo 4, and $j$ parts that are congruent to 3 modulo 4.
\end{theorem}

\noindent
We can demonstrate Theorem~\ref{BasicCorollaryComb} for special choices $i=j=1$, and $n=14$ in Table~\ref{table:TBL2}.

\begin{center}
\begin{table}[htb] \caption{$p(1,1,14)$ and $p'(1,1,14)$ with respective partitions for Theorem~\ref{BasicCorollaryComb}}\begin{tabular}{cc}
$p(1,1,14)=10$ : & $\begin{array}{c}
(13,1),\ (11,3),\ (10,3,1),\ (9,5),\ (9,3,2),\ (8,5,1),  \vspace*{.1cm}\\
(7,5,2),\ (7,4,2,1),\ (6,5,3),\ (6,4,3,1),\ .\vspace*{.1cm}
\end{array}$ \\ 
\\ [-1.5ex]
$p'(1,1,14)=10$ : & $\begin{array}{c}\vspace*{.1cm}(11,2,1),\ (10,3,1),\ (9,3,2),\  (8,3,2,1),\ (7,6,1),\\ \vspace*{.1cm}   (7,5,2),\ (7,4,2,1),\ (6,5,3),\ (6,4,3,1),\ (5,4,3,2). \vspace*{.1cm}\end{array}$ 
\end{tabular}\label{table:TBL2}  
\end{table}
\end{center}

Let $P(i,j,q)$ be the generating function for $p(i,j,n)$: \[P(i,j,q) = \sum_{n\geq 0} p(i,j,n)q^n.\]

\noindent Let $L$, $k$, $n$, $m$ be non-negative integers. We will use standard notations in \cite{theoryofpartitions} and \cite{GasperRahman}.
\begin{align*}
	(a)_L:=(a;q)_L &:= \prod_{n=0}^{L-1}(1-aq^n),\\
	(a_1,a_2,\dots,a_k;q)_L &:= (a_1;q)_L(a_2;q)_L\dots (a_k;q)_L,\\
	(a;q)_\infty &:= \lim_{L\rightarrow\infty} (a;q)_L.
\end{align*}
We define the Gaussian binomial ($q$-binomial) and $q$-trinomial coefficients respectively as
\begin{align*}
   \genfrac{[}{]}{0pt}{}{k}{n}_q &:= \left\lbrace \begin{array}{ll}\frac{(q)_k}{(q)_n(q)_{k-n}}&\text{for }k\geq n \geq 0,\\
   0&\text{otherwise,}\end{array}\right.\\
   \genfrac{[}{]}{0pt}{}{k}{n,\ m}_q &:= \genfrac{[}{]}{0pt}{}{k}{n}_q\genfrac{[}{]}{0pt}{}{k-n}{m}_q= \frac{(q)_k}{(q)_n(q)_m(q)_{k-n-m}}\text{ for }k\geq n+m\text{ and } n,m \geq 0\text{.}
\end{align*}
Next, we are going to introduce basic $q$-hypergeometric series \cite{GasperRahman}. Let $r$ and $s$ be non-negative integers and $a_1,a_2,\dots,a_r,b_1,b_2,\dots,b_s,q,$ and $z$ be variables. Basic $q$-hypergeometric series is defined as \[{}_r\phi_s\left(\genfrac{}{}{0pt}{}{a_1,a_2,\dots,a_r}{b_1,b_2,\dots,b_s};q,z\right):=\sum_{n=0}^\infty \frac{(a_1,a_2,\dots,a_r;q)_n}{(q,b_1,\dots,b_s;q)_n}\left[(-1)^nq^{n\choose 2}\right]^{1-r+s}z^n.\]

Using techniques in \cite{theoryofpartitions}, it is clear that \begin{equation}\label{genFunc1or3mod4} \frac{q^{2k^2 + (-1)^\mu k }}{(q^4;q^4)_k}
\end{equation} is the generating function for the number of partitions with $k$ distinct parts congruent to $2+(-1)^\mu$ modulo 4, where $\mu\in\{0,1\}$. Therefore, Theorem~\ref{BasicCorollaryComb} amounts to the identity
\begin{equation}\label{BasicCorollary}
P(i,j,q) = (-q^2;q^2)_\infty\frac{q^{2i^2-i}}{(q^4;q^4)_i}\frac{q^{2j^2+j}}{(q^4;q^4)_j}
\end{equation} for a variable $q$ with $|q|<1$.
Equation \eqref{BasicCorollary} is going to be discussed further in the next section.

The BG-rank of a partition $\pi$---denoted $BG(\pi)$---is defined as
\begin{equation*}
BG(\pi) :=i-j,\end{equation*}where $i$ is the number of odd-indexed odd parts and $j$ is the number of even-indexed odd parts.
Another equivalent representation of BG-rank of a partition $\pi$ comes from 2-residue diagrams. One can show that
\begin{equation*}
BG(\pi)=r_0 - r_1,
\end{equation*}
where $r_0$ is the number of $0$'s in the 2-residue diagram of $\pi$, and $r_1$ is the number of $1$'s in the 2-residue diagram. BG-rank of the partition $\pi=(12,10,7,5,2)$ in the example of Table~\ref{table:TBL1} is equal to 0.

In Section~\ref{section2}, we present and prove explicit formulas for the generating functions for the number of partitions with fixed number of odd-indexed and even-indexed odd parts. We end Section~\ref{section2} by discussing a new partition theorem which extends the work of Savage and Sills in \cite{SavageSills}. In Section~\ref{section3}, we examine various partitions with a bound on the largest part and fixed values of BG-rank. Section~\ref{section4} deals with some $q$-theoretic implications of Rogers--Szeg\H{o} polynomials. In Section~\ref{section5}, we recall the work of Boulet \cite{Boulet}, Ishikawa and Zeng \cite{Masao} and provide combinatorial interpretation for the identity for the Rogers--Szeg\H{o} polynomials in \cite{WarnaarBerkovich}. We conclude this section with a new partition theorem which utilizes alternating sum of parts statistics. In Section~\ref{section6}, we extend the work of Boulet. There, we derive new formulas for generating functions for weighted partitions with bounds on the number of parts and the size of parts. We conclude with an outlook.

\section{Generating functions and Recurrence relations}\label{section2}
%%%%%%%%%%%%%%%%%%%%%%%%%%%%%%%%%%%%%%%%%%%%%%%%%%%%%%%%%%%%%%%%%%%%%%%%%%%%%%%%%%%%%%%%%%%%%%%%%%%%%

We start by generalizing our earlier definitions. Let $p_N(i,j,n)$ be the number of partitions of $n$ counted by $p(i,j,n)$ with the extra condition that parts of the partitions, that are being counted, are less than or equal to $N$. Let $P_{N}(i,j,q)$ be the generating function for $p_N(i,j,n)$: \[P_{N}(i,j,q):=\sum_{n\geq 0} p_N(i,j,n)q^n.\] The relation between $P_N(i,j,q)$ and $P(i,j,q)$ in \eqref{BasicCorollary} is \[P(i,j,q) = \lim_{N\rightarrow\infty} P_{N}(i,j,q).\] The following theorem gives explicit formulas for the generating functions $P_N(i,j,q)$.

\begin{theorem}\label{PgenFunc} Let $N$, $i$, and $j$ be non-negative integers and $q$ be a variable then
\begin{align}
\label{Peven}P_{2N}(i,j,q) &= q^{2i^2-i+2j^2+j} (-q^2;q^2)_{N-i-j} \genfrac{[}{]}{0pt}{}{N}{i,\ j}_{q^4},\\
\label{Podd}P_{2N+1}(i,j,q) &= q^{2i^2-i+2j^2+j} (-q^2;q^2)_{N-i-j+1} \genfrac{[}{]}{0pt}{}{N+1}{i,\ j}_{q^4}\frac{1-q^{2(N+i-j+1)}}{1-q^{4(N+1)}}.
\end{align}
\end{theorem}

We prove this assertion using recurrence relations. Using the definitions of $P_N(i,j,q)$ recurrence relations are easily attained by extracting the largest parts of partitions counted by these generating functions.

\begin{lemma}\label{RecRels}  Let $N$, $i$, and $j$ be non-negative integers, $\nu\in\{0,1\}$. Then
\begin{align}
\label{recursion} P_{2N+\nu}(i,j,q) &= P_{2N+\nu-1}(i,j,q) + q^{2N+\nu} \chi(i\geq\nu)P_{2N+\nu-1}(j,i-\nu,q),
%\label{RecEven}P_{2N}(i,j,q) &= P_{2N-1}(i,j,q) + q^{2N} P_{2N-1}(j,i,q),\\
%\label{RecOdd}P_{2N+1}(i,j,q) &= P_{2N}(i,j,q) + q^{2N+1} \chi(i>0)P_{2N}(j,i-1,q),
\intertext{where $N\geq1-\nu$. The initial conditions are \[P_{0}(i,j,q) = \delta_{i,0}\delta_{j,0},\] where}
\nonumber \chi(\text{statement}) &:= \left\{\begin{array}{cl}
1 & \text{if statement is true,} \\ 
0 & \text{otherwise,} 
\end{array}  \right.\end{align} and $\delta_{i,j}:=\chi(i=j)$, the Kronecker delta function.
\end{lemma}
We remark that similar generating functions and recurrence relation were discussed in \cite{CapparelliPPR}. 

\begin{proof} Let $\nu\in\{0,1\}$, $\pi=(\lambda_1,\lambda_2,\dots,\lambda_k)$ be a partition counted by $p_{2N+\nu}(i,j,n)$ for some $k$. If $\lambda_1 <2N+\nu$, then $\pi$ is also counted by $p_{2N+\nu-1}(i,j,q)$. If $\lambda_1 = 2N+\nu$, then by extracting $\lambda_1$ we get a new partition $\pi^*=(\lambda_2,\lambda_3,\dots,\lambda_k)$ into distinct parts with largest part $\leq 2N+\nu-1$. If $\nu=0$ then the number of odd parts stay the same, but the indexing of those parts switch. If $\nu=1$, then on top of the change of parities of odd numbers, the number of odd parts in $\pi^*$ is one less than the number of odd parts in $\pi$. Therefore, $\pi^*$ is a partition that is counted by $p_{2N+\nu-1}(j,i-\nu,q)$. This proves the lemma.
\end{proof}

We need to show that the right-hand sides of \eqref{Peven} and \eqref{Podd} satisfy the same recurrence relations of Lemma~\ref{RecRels} to prove Theorem~\ref{PgenFunc}. The right-hand side of \eqref{Peven} can be rewritten for this purpose:
\begin{align*}
q^{2i^2-i+2j^2+j}& (-q^2;q^2)_{N-i-j} \genfrac{[}{]}{0pt}{}{N}{i,\ j}_{q^4}\\ &= q^{2i^2-i+2j^2+j} (-q^2;q^2)_{N-i-j} \genfrac{[}{]}{0pt}{}{N}{i,\ j}_{q^4}\frac{1-q^{2(N+i-j)}+q^{2(N+i-j)}(1-q^{2(N-i+j)})}{1-q^{4N}}\\
&= q^{2i^2-i+2j^2+j} (-q^2;q^2)_{N-i-j} \genfrac{[}{]}{0pt}{}{N}{i,\ j}_{q^4}\frac{1-q^{2(N+i-j)}}{1-q^{4N}}\\ &\hspace{1cm}+ q^{2i^2+i+2j^2-j+2N} (-q^2;q^2)_{N-i-j} \genfrac{[}{]}{0pt}{}{N}{i,\ j}_{q^4}\frac{1-q^{2(N-i+j)}}{1-q^{4N}}.
\end{align*}
This proves that the right-hand side \eqref{Peven} satisfies the recurrence relation of Lemma~\ref{RecRels} for $\nu=0$. Recurrence relation of \eqref{Podd} can be shown in the same manner. Moreover, the initial condition of Lemma~\ref{RecRels} is obviously true for the right-hand side of \eqref{Peven} and \eqref{Podd}. This finishes the proof of Theorem~\ref{PgenFunc}. 

Three comments will be noted here. First, $P_{2N+1}(i,j,q)$ is a polynomial in $q$ for all choices of $N$, $i$, and $j$ from is definition. It is also easy to see that the right-hand side of the \eqref{Podd} needs to be a polynomial by combining \eqref{Peven}, and \eqref{recursion} with $\nu=1$. Yet in its current form the right-hand side of \eqref{Podd} comes with a non-trivial rational term with no obvious cancellation. This is going to be addressed and a new representation will be given in Section~\ref{section4}. Secondly, $N\rightarrow\infty$ in either line of Theorem~\ref{PgenFunc} proves the identity \eqref{BasicCorollary}. Hence, Theorem~\ref{BasicCorollaryComb} follows. Lastly, let $k$ be a fixed non-negative integer. Taking the limit $N\rightarrow\infty$ in \eqref{Peven} and/or \eqref{Podd}, setting $j=k$ ($i=k$), summing over $i$ ($j$), and finally using $q$-binomial theorem \cite[(II.3)]{GasperRahman}, we get Theorem~\ref{generalSS}.
\begin{theorem}\label{generalSS} 
Let $k$ be a fixed non-negative integer, then
\begin{align}
\label{generalSavageSills1}\sum_{i\geq0} P(i,k,q) &= (-q^2;q^2)_\infty \frac{q^{2k^2+k}}{(q^4;q^4)_k}\sum_{i\geq0} \frac{q^{2i^2-i}}{(q^4;q^4)_i} = (-q^2;q^2)_\infty (-q;q^4)_\infty\frac{q^{2k^2+k}}{(q^4;q^4)_k}, \\
\label{generalSavageSills2}\sum_{j\geq0} P(k,j,q) &= (-q^2;q^2)_\infty \frac{q^{2k^2-k}}{(q^4;q^4)_k} \sum_{j\geq0} \frac{q^{2j^2+j}}{(q^4;q^4)_j} = (-q^2;q^2)_\infty (-q^3;q^4)_\infty\frac{q^{2k^2-k}}{(q^4;q^4)_k}.
\end{align}\end{theorem}

Next, we can compare combinatorial interpretations of extremes of \eqref{generalSavageSills1} and \eqref{generalSavageSills2}. The sum on the left-hand side of the identity \eqref{generalSavageSills1} (or \eqref{generalSavageSills2}) gives us the generating function for number of partitions into distinct parts with $k$ even-indexed (or odd-indexed) odd parts. On the right-hand side of \eqref{generalSavageSills1} (or \eqref{generalSavageSills2}) we have the generating function for the number of partitions into distinct parts with exactly $k$ parts congruent to 1 (or 3) modulo 4. We rewrite these interpretations together.

\begin{theorem}\label{generalSSCombinatorial} For a fixed non-zero integer $k$, the number of partitions of $n$ into distinct parts where there are $k$ odd-indexed (even-indexed) odd parts is equal to the number of partitions of $n$ into distinct parts where there are exactly $k$ parts congruent to 1 (3) modulo 4.\end{theorem}

Special case $k=0$ in the Theorem~\ref{generalSS} and Theorem~\ref{generalSSCombinatorial} is well known in literature. Setting $k=0$, we can rewrite products in \eqref{generalSavageSills1} and \eqref{generalSavageSills2} as
\begin{align}\label{prodSills}(-q^2;q^2)_\infty (-q~;q^4)_\infty &= \frac{(-q;q^4)_\infty (q;q^4)_\infty}{(q^2;q^4)_\infty(q;q^4)_\infty} =\frac{(q^2;q^8)_\infty}{(q,q^2,q^5,q^6;q^8)_\infty} =  \frac{1}{(q,q^5,q^6;q^8)_\infty},\\
\intertext{and}
\label{prodSills2}(-q^2;q^2)_\infty (-q^3;q^4)_\infty &= \frac{1}{(q^2,q^3,q^7;q^8)_\infty},
\end{align} 
where we use the Euler Theorem. Far right products of \eqref{prodSills} and \eqref{prodSills2} appear in the little\footnote{This terminology was introduced by Alladi \cite{Alladi}. It is used to distinguish between (little) G\"ollnitz and (big) G\"ollnitz partition theorems.} G\"ollnitz identities \cite{Gollnitz}. Little G\"ollnitz identities have the combinatorial counterparts.

\begin{theorem}[G\"ollnitz]\label{realLittleGollnitz} The number of partitions of $n$ into parts (greater than 1) differing by at least 2, and no consecutive odd parts appear in the partitions is equal to the number of partitions of $n$ into parts congruent to 1, 5 or 6 modulo 8 (2, 3 or 7 modulo 8).
\end{theorem}

Comparing Theorem~\ref{generalSSCombinatorial} with $k=0$ and \eqref{prodSills}, \eqref{prodSills2} we arrive at the ``new little G\"ollnitz'' theorems of Savage and Sills \cite{SavageSills}.

\begin{theorem}[Savage, Sills]\label{NewLittleGollnitz} The number of partitions of $n$ into distinct parts where odd-indexed (even-indexed) parts are even is equal to the number of partitions of $n$ into parts congruent to 2, 3 or 7 modulo 8 (1, 5 or 6 modulo 8).\end{theorem}

\noindent
An interested reader is invited to examine \cite{Alladi} for another companion to the G\"ollnitz identities.

\section{Partitions with fixed value of BG-rank}\label{section3}
%%%%%%%%%%%%%%%%%%%%%%%%%%%%%%%%%%%%%%%%%%%%%%%%%%%%%%%%%%%%%%%%%%%%%%%%%%%%%%%%%%%%%%%%%%%%%%%%%%%%%

The explicit generating function formulas for $P_N(i,j,q)$, given in Theorem~\ref{PgenFunc}, opens the door to various applications and interesting interpretations. We can start by setting $i=j+k$ for any integer $k$. $P_N(j+k,j,q)$ is the generating function for number of partitions into distinct parts $\leq N$ with $j$ even-indexed odd parts and BG-rank equal to $k$. By summing these functions over $j$, we lift the restriction on the even-indexed odd parts. Let $B_N(k,q)$ denote the generating function for number of partitions into distinct parts less than or equal to $N$ with BG-rank equal to $k$. Then, \begin{equation}\label{BGdefinition}B_{N}(k,q):=\sum_{j\geq 0} P_{N}(j+k,j,q).\end{equation} We have a new combinatorial result.

\begin{theorem}\label{distinctBG} Let $N$ and $j$ be non-negative integers, $k$ be any integer. Then\begin{equation}\label{distinctBGEQN}
B_{2N+\nu}(k,q) = q^{2k^2-k}  \genfrac{[}{]}{0pt}{}{2N+\nu}{N+k}_{q^2},\end{equation}
where $\nu\in\{0,1\}$.
\end{theorem}

\begin{proof} This identity is a consequence of the $q$-Gauss identity \cite[(II.8)]{GasperRahman}. We can outline the proof as follows. Using the definition of $P_{2N+\nu}(j+k,j,q)$, (I.10), (I.25) in \cite{GasperRahman} as needed we come to the following
\begin{align*}
B_{2N+\nu}(k,q)&=q^{2k^2-k} \frac{(q^4;q^4)_N(1-\nu q^{2(N+k+1)})}{(q^4;q^4)_k(q^2;q^2)_{N-k+\nu}}\times\\\nonumber&\hspace*{2cm} {}_2\phi_1\left( \genfrac{}{}{0pt}{}{q^{-2(N-k+\nu)},\; q^{-2(N-k+\nu-1)}}{q^{4(k+1)}}; q^4,\ q^{4(N+\nu)+2} \right).
\end{align*}
Applying the $q$-Gauss identity and rewriting infinite products in a non-trivial fashion using (I.5) in \cite{GasperRahman} repeatedly yields
\begin{equation*}
B_{2N+\nu}(k,q)= q^{2k^2-k} \genfrac{[}{]}{0pt}{}{2N+\nu}{N+k}_{q^2}.
\end{equation*}
\end{proof}

Theorem~\ref{distinctBG} can also be used to prove the similar result for partitions (not necessarily in distinct parts) with the same type of bounds on the largest part and fixed BG-rank. Let $\wTilde{B}_{N}(k,q)$ be the generating function for number of partitions into parts less than or equal to $N$ with BG-rank equals $k$.

\begin{theorem}\label{unrestrictedBG} Let $N$ be a non-negative integer, $k$ be any integer. Then
\begin{equation*} \wTilde{B}_{2N+\nu}(k,q)=  \frac{q^{2k^2-k}}{(q^2;q^2)_{N+k}(q^2;q^2)_{N-k+\nu}},\end{equation*}
where $\nu\in\{0,1\}$.
\end{theorem}

The proof of Theorem~\ref{unrestrictedBG} comes from the combinatorial bijection of extracting doubly repeating parts to partitions. Let $U_{N,k}$ be the set of partitions with parts less than or equal to $N$ and BG-rank equal to $k$. Let $D_{N,k}$ be the set of partitions into distinct parts less than or equal to $N$ with BG-rank being equal to $k$. Let $E_{N}$ be the set of partitions with parts less than or equal to $N$, whose parts appear an even number of times. Define bijection $\rho:U_{N,k}\rightarrow D_{N,k}\times E_{N}$ where $\rho_{N,k}(\pi)=(\pi',\pi^*)$, where this bijection is established as the row extraction of two parts of same size at once from a given partition repeatedly until there are no more repeating parts in the outcome partition $\pi'$. The extracted parts are collected in the partition $\pi^*$. Note that the number of odd parts in $\pi$ and $\pi'$ might be different $BG(\pi)=BG(\pi')$. Table~\ref{table:TBL3} is an example of this map with $\pi=(7,5,5,5,4,4,2)$, so $\rho_{N,k}(\pi)=(\pi',\pi^*)=((7,5,2),(5,5,4,4))$ with $BG(\pi)=BG(\pi')=0$ for any $N\geq 7$ and $k=0$.

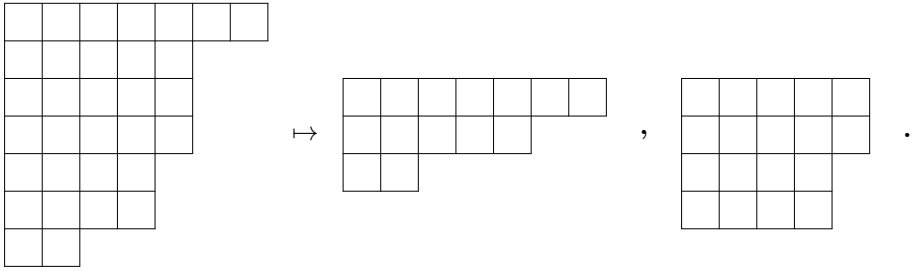
\begin{center} 
\definecolor{cqcqcq}{rgb}{0.75,0.75,0.75}
\begin{table}[htb]\caption{$\rho_{N,k}((7,5,5,5,4,4,2)) = ((7,5,2),(5,5,4,4))$}
\begin{tikzpicture}[line cap=round,line join=round,x=0.5cm,y=0.5cm]
\draw [color=cqcqcq,dash pattern=on 1pt off 1pt, xstep=0.5cm,ystep=0.5cm] ;
\clip(0.5,0.5) rectangle (26,8.5);
\draw [line width=.25pt] (3,1)-- (1,1);
\draw [line width=.25pt] (1,1)-- (1,2);
\draw [line width=.25pt] (5,2)-- (1,2);
\draw [line width=.25pt] (1,3)-- (5,3);
\draw [line width=.25pt] (1,4)-- (5,4);
\draw [line width=.25pt] (1,5)-- (6,5);
\draw [line width=.25pt] (1,6)-- (6,6);
\draw [line width=.25pt] (1,7)-- (6,7);
\draw [line width=.25pt] (6,4)-- (5,4);
\draw [line width=.25pt] (8,8)-- (1,8);
\draw [line width=.25pt] (8,7)-- (6,7);
\draw [line width=.25pt] (1,8)-- (1,2);
\draw [line width=.25pt] (8,8)-- (8,7);
\draw [line width=.25pt] (7,8)-- (7,7);
\draw [line width=.25pt] (6,8)-- (6,7);
\draw [line width=.25pt] (5,8)-- (5,2);
\draw [line width=.25pt] (6,7)-- (6,4);
\draw [line width=.25pt] (4,8)-- (4,2);
\draw [line width=.25pt] (3,8)-- (3,1);
\draw [line width=.25pt] (2,8)-- (2,1);

\draw [line width=.25pt] (10,6)-- (17,6);
\draw [line width=.25pt] (10,5)-- (17,5);
\draw [line width=.25pt] (10,4)-- (15,4);
\draw [line width=.25pt] (10,3)-- (12,3);
\draw [line width=.25pt] (10,3)-- (10,6);
\draw [line width=.25pt] (11,6)-- (11,3);
\draw [line width=.25pt] (12,6)-- (12,3);
\draw [line width=.25pt] (13,6)-- (13,4);
\draw [line width=.25pt] (14,6)-- (14,4);
\draw [line width=.25pt] (15,6)-- (15,4);
\draw [line width=.25pt] (16,6)-- (16,5);
\draw [line width=.25pt] (17,6)-- (17,5);
\draw [line width=.25pt] (19,6)-- (24,6);
\draw [line width=.25pt] (24,5)-- (19,5);
\draw [line width=.25pt] (19,4)-- (24,4);
\draw [line width=.25pt] (19,3)-- (23,3);
\draw [line width=.25pt] (19,2)-- (23,2);
\draw [line width=.25pt] (19,6)-- (19,2);
\draw [line width=.25pt] (24,6)-- (24,4);
\draw [line width=.25pt] (23,2)-- (23,6);
\draw [line width=.25pt] (22,6)-- (22,2);
\draw [line width=.25pt] (21,6)-- (21,2);
\draw [line width=.25pt] (20,6)-- (20,2);

\draw (9,4.5) node[anchor=center] {$\mapsto$};
\draw (18,4.5) node[anchor=center] {\textbf{,}};

\draw (25,4.5) node[anchor=center] {\textbf{.}};

%\draw (4.5,9) node[anchor=center] {$\pi$};
%\draw (9,9) node[anchor=center] {$\mapsto$};
%\draw (10,9) node[anchor=center] {$($};
%\draw (14,9) node[anchor=center] {$\pi'$};
%\draw (18,9) node[anchor=center] {\textbf{,}};
%\draw (22,9) node[anchor=center] {$\pi^*$};
%\draw (25,9) node[anchor=center] {$)$};

\end{tikzpicture}\label{table:TBL3}
\end{table}
\end{center}

The generating function for number of partitions from the set $E_{2N+\nu}$, where all parts appear an even number of times and are less than or equal to $2N+\nu$ is \[\frac{1}{(q^2;q^2)_{2N+\nu}},\] where $\nu\in\{0,1\}$. Therefore, keeping the bijection $\rho_{2N+\nu,k}$ above in mind, the generating function for number of partitions with BG-rank equal to $k$ and the bound on the largest part being $2N+\nu$ is the product \begin{equation}\label{BGunrestrictedToDistinct}\wTilde{B}_{2N+\nu}(k,q)=\frac{B_{2N+\nu}(k,q)}{(q^2;q^2)_{2N+\nu}},\end{equation} proving Theorem~\ref{unrestrictedBG}. 

The combinatorial interpretation coming from \eqref{distinctBGEQN} was previously unknown. Recall that the expression \[q^{2k^2-k}  \genfrac{[}{]}{0pt}{}{2N+\nu}{N+k}_{q^2},\] is the generating function for number of partitions with BG-rank equal to $k$ and the largest part $\leq 2N+\nu$, where $\nu\in\{0,1\}$. This relation gives a combinatorial explanation of a well known identity:
\begin{theorem}\label{BG_Rogers_Szego} Let $\nu\in\{0,1\}$, and let $N+\nu$ be a positive integer. Then,\begin{equation}\label{BGgen}\sum_{k=-N}^{N+\nu}q^{2k^2-k} \genfrac{[}{]}{0pt}{}{2N+\nu}{N+k}_{q^2}= (-q;q)_{2N+\nu}.\end{equation}\end{theorem}
\noindent

\noindent
It is clear that summing $B_N(k,q)$, defined in \eqref{BGdefinition}, over all possible BG-ranks yield the generating function for the number of partitions into distinct parts $\leq N$, yielding \eqref{BGgen}.

Identity \eqref{BGgen} was discussed in \cite[4.2]{Gaussian} by Andrews. He showed that \eqref{BGgen} is equivalent to an identity for Rogers--Szeg\H{o} polynomials. Recall that the Rogers--Szeg\H{o} polynomials are defined as
\begin{align}\label{qhermite}H_N(z,q)&:=\sum_{l=0}^N \genfrac{[}{]}{0pt}{}{N}{l}_{q}z^l.\\ \intertext{Then we have the identity \cite{Gaussian}} 
\label{QHsum}H_{2N+\nu}(q,q^2) &= (-q;q)_{2N+\nu},\end{align}
where $\nu\in\{0,1\}$. In oder to show the equivalence of \eqref{BGgen} and \eqref{QHsum} we use  \begin{equation}\label{chanceofbase}\genfrac{[}{]}{0pt}{}{n+m}{n}_{q^{-1}} = q^{-nm}\genfrac{[}{]}{0pt}{}{n+m}{n}_{q},\end{equation} where  $n$ and $m$ are positive integers.

In \eqref{BGgen}, first we let $q\mapsto q^{-1}$. Then we change the $q$-binomial term using \eqref{chanceofbase} on the left-hand side, and rewrite the right-hand side as \[\left(-1/q;1/q\right)_{2N+\nu} = q^{-(N+\nu)(2N+1)}(-q;q)_{2N+\nu}.\] Multiplying both sides with $q^{(N+\nu)(2N+1)}$ and changing the summation variable in the equation \eqref{BGgen} with $k\mapsto N+(-1)^\nu k +\nu$. In $\nu=1$ case, in addition, we change the order of summation. In this way we arrive at \eqref{QHsum}. 

Theorem~\ref{unrestrictedBG} and Theorem~\ref{BG_Rogers_Szego} is enough to prove the following, combinatorially anticipated, corollary.

\begin{corollary}\label{SumUnrestrictedBG} Let $N$ be a non-negative integer. Then \begin{equation}\label{unrestrictedBGtermWithoutT}\sum_{k=-N}^{N+\nu} \frac{q^{2k^2-k}}{(q^2;q^2)_{N+k}(q^2;q^2)_{N-k+\nu}} = \frac{1}{(q;q)_{2N+\nu}},\end{equation} where $\nu\in\{0,1\}$.
\end{corollary}

\section{Some Implications of the formula for Rogers--Szeg\H{o} polynomials}\label{section4}
%%%%%%%%%%%%%%%%%%%%%%%%%%%%%%%%%%%%%%%%%%%%%%%%%%%%%%%%%%%%%%%%%%%%%%%%%%%%%%%%%%%%%%%%%%%%%%%%%%%%%%
%Using $P_N(i,j,q)$, we can write the generating function for the weighted count of partitions into distinct parts, where we also keep track of the odd-indexed and even-indexed odd parts in the counted partitions. 

Let $q$, $t$, and $z$ be variables. It is obvious that the double sum
\begin{equation}\label{Double}
\sum_{i,j\geq 0} P_{N} (i,j,q)t^iz^j
\end{equation}
is the generating function for the number of partitions into distinct parts, where exponents of $t$ and $z$ keep track of the odd-indexed and the even-indexed odd parts, respectively. We would like to show that the double sum \eqref{Double} can be written as a single-fold sum.
\begin{theorem}\label{DoubletoSingle} Let $N$ be a non-negative integer. Then
\[\sum_{i,j\geq 0} P_{2N+\nu} (i,j,q)t^iz^j = \sum_{i\geq 0} \genfrac{[}{]}{0pt}{}{N}{i}_{q^4} (-qt;q^4)_{N-i+\nu}(-qz;q^4)_{i} q^{2i} ,\] where $\nu\in\{0,1\}.$
\end{theorem}

\noindent Proving this theorem relies on the following new identity.

\begin{theorem}\label{New5.2} For non-negative integers $N$, $i$, $j$ and a variable $q$
\begin{equation}
\label{new5.2} \sum_{l=0} ^N \genfrac{[}{]}{0pt}{}{N}{l}_{q^2} \genfrac{[}{]}{0pt}{}{l+\nu}{i}_{q^2}\genfrac{[}{]}{0pt}{}{N-l}{j}_{q^2} q^{N-l-j} = (-q;q)_{N-i-j+\nu} \genfrac{[}{]}{0pt}{}{N+\nu}{i,\ j}_{q^2}\frac{1-\nu q^{N+i-j+1}}{1-\nu q^{2(N+1)}},\end{equation} where $\nu\in\{0,1\}$.
\end{theorem}

We prove these identities by appeal to Rogers--Szeg\H{o} polynomial summation formula \eqref{QHsum}. We divide both sides of \eqref{new5.2} with \[\genfrac{[}{]}{0pt}{}{N+\nu}{i,\ j}_{q^2}\frac{1}{1-\nu q^{2(N+1)}}\] and simplify the factorials. Using, \[H_{N-i-j+\nu}(q,q^2) = (-q;q)_{N-i-j+\nu}\hspace{.3mm} , \] from \eqref{QHsum}, for $\nu\in\{0,1\}$ we get the final result. 

The identity \eqref{new5.2} with $\nu=1$ provides an explicit polynomial formula for \eqref{Podd}. Furthermore, \eqref{new5.2} is enough to demonstrate Theorem~\ref{DoubletoSingle}. Let $\llfloor t^iz^j\rrfloor$ denote the coefficient of $t^iz^j$ term of the series this notation precedes.

\begin{proof}(Theorem~\ref{DoubletoSingle}) Use the $q$-binomial theorem, \cite[(II.3)]{GasperRahman}, to explicitly write the powers of $t$ and $z$. Extracting the coefficient of $t^iz^j$ we get\[\llfloor t^iz^j\rrfloor \sum_{l\geq 0} \genfrac{[}{]}{0pt}{}{N}{l}_{q^4} (-qt;q^4)_{N-l+\nu}(-qz;q^4)_{l} q^{2l} = q^{\omega(i,j)}\sum_{l=0} ^N \genfrac{[}{]}{0pt}{}{N}{l}_{q^4} \genfrac{[}{]}{0pt}{}{l+\nu}{i}_{q^4}\genfrac{[}{]}{0pt}{}{N-l}{j}_{q^4} q^{2(N-l-j)}, \] where $\omega(i,j)={2i^2-i+2j^2+j}$ and $\nu\in\{0,1\}$. This gives us the result by comparing \eqref{new5.2} with $q^2\mapsto q^4$ and equalities \eqref{Peven} and \eqref{Podd}.
\end{proof} 

Theorem~\ref{DoubletoSingle} yields interesting corollaries. Setting $(t,z)=(0,1)$ and $(1,0)$ gives us the new $q$-series identities of Corollary~\ref{corollarySSGENERAL}.

\begin{corollary}\label{corollarySSGENERAL} Let $N$ and $k$ be non-negative integers. Then
\begin{align}
\label{qNewLittleGollnitz1}\sum_{k\geq 0}  \genfrac{[}{]}{0pt}{}{N+\nu}{k}_{q^4} (-q^2;q^2)_{N-k+\nu} q^{2k^2+k}\frac{1-\nu q^{2(N-k+1)}}{1-\nu q^{4(N+1)}}&= \sum_{k\geq0}  \genfrac{[}{]}{0pt}{}{N}{k}_{q^4} (-q;q^4)_k q^{2k},\\
\label{qNewLittleGollnitz2}\sum_{k\geq 0}  \genfrac{[}{]}{0pt}{}{N+\nu}{k}_{q^4} (-q^2;q^2)_{N-k+\nu}q^{2k^2-k}\frac{1-\nu q^{2(N+k+1)}}{1-\nu q^{4(N+1)}} &= \sum_{k\geq0}  \genfrac{[}{]}{0pt}{}{N}{k}_{q^4} (-q;q^4)_{N-k+\nu} q^{2k},
\end{align} where $\nu\in\{0,1\}$.
\end{corollary}

It is amusing to observe that there is no dependence on $\nu$ on the right-hand side of \eqref{qNewLittleGollnitz1}. With a later theorem in Section~\ref{section6}, we can also prove the following.

\begin{theorem} Let $N$ be a non-negative integer, then
\begin{equation}\label{connect} \sum_{k\geq0}  \genfrac{[}{]}{0pt}{}{N}{k}_{q^4} (-q;q^4)_k q^{2k} = \sum_{k\geq0} \genfrac{[}{]}{0pt}{}{N}{k}_{q^4} (-q^3;q^4)_{N-k} q^{2k}.\end{equation}
\end{theorem}
This theorem is easily proven by choosing $(a,b,c,d) = (qt,q/t,qz,q/z)$ in \eqref{finiteBouletPSI} and setting $(t,z)=(0,1)$. Note that, the right-hand sides of \eqref{qNewLittleGollnitz2} and \eqref{connect} are very similar. In fact, the choice $(a,b,c,d) = (qt,q/t,qz,q/z)$ and then $(t,z)=(1,0)$ in \eqref{finiteBouletPSI} yields the right-hand side of \eqref{qNewLittleGollnitz2}. Johann Cigler brought to our attention that \eqref{connect} is an easy corollary of a more general identity.

\begin{theorem} Let $N$ be a non-negative integer, then 
\begin{equation*} \sum_{k\geq0}  \genfrac{[}{]}{0pt}{}{N}{k}_{q} (y;q)_k z^{k} = \sum_{k\geq0} \genfrac{[}{]}{0pt}{}{N}{k}_{q} (yz;q)_{N-k} z^{k}.\end{equation*}
\end{theorem}

\noindent
This theorem can easily be proven by appeal to $q$-binomial theorem \cite[II.3]{GasperRahman}.

\section{Partitions with Boulet-Stanley weights}\label{section5}
%%%%%%%%%%%%%%%%%%%%%%%%%%%%%%%%%%%%%%%%%%%%%%%%%%%%%%%%%%%%%%%%%%%%%%%%%%%%%%%%%%%%%%%%%%%%%%%%%%%%%

In \cite{Stanley}, Stanley made a suggestion for suitable weights to be used on diagrams. Boulet, in \cite{Boulet}, extensively utilized this suggestion on what we will denote a four-variable decoration of Ferrers diagram of a partition. We can decorate any Ferrers diagram of a partition with variables $a$, $b$, $c$, and $d$. We fill the boxes on the odd-indexed rows with alternating variables $a$ and $b$ starting from $a$, and boxes on the even-indexed rows filled with alternating variables $c$ and $d$ starting from $c$. 

One can define a weight on these four-variable decorated diagrams as \[\omega_\pi(a,b,c,d) = a^{\#a}b^{\#b}c^{\#c}d^{\#d},\] where $\#a$ denotes the number of boxes decorated with variable $a$ in partition $\pi$'s  four-variable decorated diagram, etc. One example of a four-variable decoration and the weight is given in Table~\ref{table:TBL4}.

\begin{center}
\begin{table}[htb]\caption{Four-variable decorated Ferrers diagram of the partition $\pi=(12,10,7,5,2)$}
\begin{tikzpicture}[line cap=round,line join=round,x=0.5cm,y=0.5cm]
\clip(0.5,0) rectangle (14.5,6.5);
\draw [line width=.25pt] (2,1)-- (4,1);
\draw [line width=.25pt] (2,2)-- (7,2);
\draw [line width=.25pt] (2,3)-- (9,3);
\draw [line width=.25pt] (2,4)-- (12,4);
\draw [line width=.25pt] (14,5)-- (2,5);
\draw [line width=.25pt] (2,6)-- (14,6);
\draw [line width=.25pt] (14,6)-- (14,5);
\draw [line width=.25pt] (13,6)-- (13,5);
\draw [line width=.25pt] (12,6)-- (12,4);
\draw [line width=.25pt] (11,6)-- (11,4);
\draw [line width=.25pt] (10,6)-- (10,4);
\draw [line width=.25pt] (9,6)-- (9,3);
\draw [line width=.25pt] (8,6)-- (8,3);
\draw [line width=.25pt] (7,2)-- (7,6);
\draw [line width=.25pt] (6,6)-- (6,2);
\draw [line width=.25pt] (5,6)-- (5,2);
\draw [line width=.25pt] (4,1)-- (4,6);
\draw [line width=.25pt] (3,6)-- (3,1);
\draw [line width=.25pt] (2,6)-- (2,1);
\draw (2.5,5.5) node[anchor=center] {$a$};
\draw (3.5,5.5) node[anchor=center] {$b$};
\draw (4.5,5.5) node[anchor=center] {$a$};
\draw (5.5,5.5) node[anchor=center] {$b$};
\draw (6.5,5.5) node[anchor=center] {$a$};
\draw (7.5,5.5) node[anchor=center] {$b$};
\draw (8.5,5.5) node[anchor=center] {$a$};
\draw (9.5,5.5) node[anchor=center] {$b$};
\draw (10.5,5.5) node[anchor=center] {$a$};
\draw (11.5,5.5) node[anchor=center] {$b$};
\draw (12.5,5.5) node[anchor=center] {$a$};
\draw (13.5,5.5) node[anchor=center] {$b$};

\draw (2.5,4.5) node[anchor=center] {$c$};
\draw (3.5,4.5) node[anchor=center] {$d$};
\draw (4.5,4.5) node[anchor=center] {$c$};
\draw (5.5,4.5) node[anchor=center] {$d$};
\draw (6.5,4.5) node[anchor=center] {$c$};
\draw (7.5,4.5) node[anchor=center] {$d$};
\draw (8.5,4.5) node[anchor=center] {$c$};
\draw (9.5,4.5) node[anchor=center] {$d$};
\draw (10.5,4.5) node[anchor=center] {$c$};
\draw (11.5,4.5) node[anchor=center] {$d$};

\draw (2.5,3.5) node[anchor=center] {$a$};
\draw (3.5,3.5) node[anchor=center] {$b$};
\draw (4.5,3.5) node[anchor=center] {$a$};
\draw (5.5,3.5) node[anchor=center] {$b$};
\draw (6.5,3.5) node[anchor=center] {$a$};
\draw (7.5,3.5) node[anchor=center] {$b$};
\draw (8.5,3.5) node[anchor=center] {$a$};

\draw (2.5,2.5) node[anchor=center] {$c$};
\draw (3.5,2.5) node[anchor=center] {$d$};
\draw (4.5,2.5) node[anchor=center] {$c$};
\draw (5.5,2.5) node[anchor=center] {$d$};
\draw (6.5,2.5) node[anchor=center] {$c$};

\draw (2.5,1.5) node[anchor=center] {$a$};
\draw (3.5,1.5) node[anchor=center] {$b$};

\draw (6.5,0.3) node[anchor=center] {$\omega_\pi(a,b,c,d) = a^{11}b^{10}c^{8}d^7$};
\end{tikzpicture}\label{table:TBL4}
\end{table}
\end{center}

The generating function for the weighted count of Ferrers diagrams of partitions from a set $S$ with weight $\omega_\pi(a,b,c,d)$ is \[\sum_{\pi\in S} \omega_\pi(a,b,c,d).\]
Let $\rho_N$ be the obvious bijective map $U_N\mapsto D_N\times E_N$. Let $U_N$ be the set of partitions with parts less than or equal to $N$, and $D_N$ be the set of partitions into distinct parts with parts less than or equal to $N$. Recall that $E_N$ is the set of partitions into parts less than or equal to $N$, where every part repeats an even number of times. Let $\rho: U\mapsto D\times E$ where $\rho$, $U$, $D$, and $E$ are the same sets as above, where we remove the bound on the largest part $N$.

In \cite{Boulet}, Boulet proved identities for the generating functions for weighted count of partitions with four-decorated Ferrers diagrams from the sets $D$ and $U$.
\begin{theorem}[Boulet] \label{boulet}For variables $a$, $b$, $c$, and $d$ and $Q:=abcd$, we have\begin{align} \label{PSI}\Psi(a,b,c,d) &:= \sum_{\pi\in D} \omega_\pi(a,b,c,d) = \frac{(-a,-abc;Q)_\infty}{(ab;Q)_\infty},\\
\label{PHI}\Phi(a,b,c,d) &:= \sum_{\pi\in U} \omega_\pi(a,b,c,d) = \frac{(-a,-abc;Q)_\infty}{(ab,ac,Q;Q)_\infty}.
\end{align}\end{theorem} 
The transition from \eqref{PHI} to \eqref{PSI} can be done with the aid of the bijective map $\rho$. The generating function for the weighted count of four-variable Ferrers diagrams which exclusively have an even number of rows of the same length is \[\frac{1}{(ac,Q;Q)_\infty}.\] Hence, we get  \begin{equation}\label{PSItoPHI}
\Phi(a,b,c,d) = \frac{\Psi(a,b,c,d)}{(ac,Q;Q)_\infty},
\end{equation}
which is similar to \eqref{BGunrestrictedToDistinct}.

\noindent It should be noted that these generating functions are consistent with the ordinary generating functions for number of partitions when all variables are selected to be $q$. For example, \[\Phi(q,q,q,q)=\frac{1}{(q;q)_\infty}.\]

\noindent Define the generating functions \begin{align}
 \label{PSI_N}\Psi_N(a,b,c,d) &:= \sum_{\pi\in D_N} \omega_\pi(a,b,c,d),\\
 \label{PHI_N}\Phi_N(a,b,c,d) &:= \sum_{\pi\in U_N} \omega_\pi(a,b,c,d),
 \end{align}
which are finite analogues of Boulet's generating functions for the weighted count of four-variable decorated Ferrers diagrams. In \cite{Masao}, Ishikawa and Zeng write explicit formulas for \eqref{PSI_N} and \eqref{PHI_N}. 

\begin{theorem}[Ishikawa, Zeng]\label{MasaoThm} For a non-zero integer $N$, variables $a$, $b$, $c$, and $d$, we have
\begin{align}
\label{PSI2Nnu}\Psi_{2N+\nu}(a,b,c,d) &= \sum_{i=0}^N \genfrac{[}{]}{0pt}{}{N}{i}_{Q} (-a;Q)_{N-i+\nu}(-c;Q)_{i} (ab)^{i},\\
\label{PHI2Nnu}\Phi_{2N+\nu}(a,b,c,d) &=\frac{1}{(ac;Q)_{N+\nu}(Q;Q)_{N}}\sum_{i=0}^N \genfrac{[}{]}{0pt}{}{N}{i}_{Q} (-a;Q)_{N-i+\nu}(-c;Q)_{i} (ab)^{i},
\end{align}
where $\nu\in\{0,1\}$ and $Q=abcd$.
\end{theorem}
\noindent
We would like to point out the special case of \eqref{PHI2Nnu} with $(a,b,c,d) =(qzy,~qy/z,~qz/y,~q/zy)$ was first discovered and proven by Andrews in \cite{AndrewsStanley}. 
We note in passing that \eqref{PSI2Nnu} together with the $q$-binomial theorem \cite[II.3]{GasperRahman} yields 
\begin{theorem} For $\nu\in\{0,1\}$,\[\sum_{N\geq 0} \frac{x^N}{(Q;Q)}_N\Psi_{2N+\nu}(a,b,c,d) = (1+a\nu)\frac{(-axQ^\nu;Q)_\infty (-abcx;Q)_\infty}{(x;Q)_\infty(abx;Q)_\infty}, \] where $Q=abcd$.\end{theorem}

Similar to \eqref{PSItoPHI}, the connection between \eqref{PSI2Nnu} and \eqref{PHI2Nnu} can be obtained by means of the bijection $\rho_{N}$. In this way we have \begin{equation}\label{PSItoPHI_Bounded}\Phi_{2N+\nu}(a,b,c,d) = \frac{\Psi_{2N+\nu}(a,b,c,d)}{(ac;Q)_{N+\nu}(Q;Q)_{N},}\end{equation} where $N$ is a non-negative integer and $\nu\in\{0,1\}$. 

We can now rewrite Theorem~\ref{DoubletoSingle} as
\begin{equation}\label{genFuncDistinct}\Psi_{2N+\nu}(qt,q/t,qz,q/z)= \sum_{i= 0}^N \genfrac{[}{]}{0pt}{}{N}{i}_{q^4} (-qt;q^4)_{N-i+\nu}(-qz;q^4)_{i} q^{2i}.\end{equation}
This may be viewed as a special case of Theorem~\ref{MasaoThm}. 

The sum similar to that on the right-hand side of \eqref{PSI2Nnu} was seen in the literature before. In fact, Berkovich and Warnaar \cite{WarnaarBerkovich} utilized that sum in their identity for the Rogers--Szeg\H{o} polynomials.

\begin{theorem} [Berkovich, Warnaar]\label{HermiteBerkovichTHM} Let $N$ be a non-negative integer. Then the Rogers--Szeg\H{o} polynomials can be expressed as
\begin{equation}\label{HermiteBerkovich}
H_{2N+\nu}(zq,q^2) = \sum_{l=0}^{N}\genfrac{[}{]}{0pt}{}{N}{l}_{q^4} (-zq;q^4)_{N-l+\nu}(-q/z;q^4)_l (zq)^{2l}  
\end{equation} for $\nu\in\{0,1\}$.
\end{theorem}

\noindent We would like to mention that recently Cigler provided a new proof of \eqref{HermiteBerkovich} in \cite{Cigler}.

Let $\pi = (\lambda_1,\lambda_2,\dots,\lambda_k)$ be a partition. We want to define the weights $|\pi| = \lambda_1+\lambda_2+\dots+\lambda_k$, called the norm of the partition $\pi$, and $\gamma(\pi) = \lambda_1-\lambda_2+\lambda_3-\lambda_4+\dots+(-1)^{k+1}\lambda_k$, the alternating sum of parts of the partition $\pi$.

We see that the right-hand sides of \eqref{PSI2Nnu} and \eqref{HermiteBerkovich} coincide with the particular choice $(a,b,c,d)=(zq,zq,q/z,q/z)$ in \eqref{PSI2Nnu},
\begin{equation}\label{RogersSzegoToPSI}H_{N}(zq,q^2) = \Psi_{N}(zq,zq,q/z,q/z).\end{equation}

%\noindent We can also rewrite the $\Psi_N(zq,zq,q/z,q/z)$ function as a generating function for the weighted count of partitions. Recall that $D_N$ is the set of partitions into distinct parts $\leq N$. The generating function for partitions into distinct parts $\leq N$ counted with weights, where the exponent of $z$ keeps track of the alternating sums of the partitions, and the exponent of $q$ is the norm of these partitions is given by
Obviously $\Psi_N(zq,zq,q/z,q/z)$ is the generalting function for the number of partitions into distinct parts $\leq N$, where exponent of $z$ is the alternating sum of the parts of partition. That is
\begin{equation*}
\Psi_{N}(zq,zq,q/z,q/z)=\sum_{\pi\in D_{N}} q^{|\pi|}\ z^{\gamma(\pi)}.
\end{equation*}
Therefore, extraction of the coefficient of $z^k$ would give us the generating function for the number of partitions with the alternating sums of parts equal to $k$. Using definition for $H_{N}(zq,q^2)$, given in \eqref{qhermite}, one can easily extract the coefficient of $z^k$ for a fixed non-negative integer $k$. This way we get \begin{equation}\label{extraction}q^k \genfrac{[}{]}{0pt}{}{N}{k}_{q^2}=\sum_{\begin{array}{c}
\pi\in D_{N},\\ \gamma(\pi)=k\end{array}} q^{|\pi|}. \end{equation} This can be interpreted as

\begin{theorem}\label{combHermiteSum} Let $N$, $n$, and $k$ be non-negative integers. The number of partitions of $n$ into solely $k$ odd parts $\leq 2N-2k+1$ is equal to the number of partitions of $n$ into distinct parts $\leq N$ with the alternating sum of parts being equal to $k$.
\end{theorem}

This result can be explained with the aid of the Sylvester bijection, as in Bressoud's book (\cite{Bressoud}, p.52, ex. 2.2.5). Thus, Theorem~\ref{HermiteBerkovichTHM} may be interpreted as an analytical proof of Theorem~\ref{combHermiteSum}.

Moreover, the connection \eqref{PSItoPHI_Bounded} used in \eqref{extraction} yields \begin{equation}\label{extractionPHI}\llfloor z^k \rrfloor \Phi_{N}(zq,zq,q/z,q/z) = \frac{q^k}{(q^2;q^2)_{k}(q^2;q^2)_{N-k}},\end{equation} where $\llfloor z^k \rrfloor \Phi_{N}(zq,zq,q/z,q/z)$ denotes the coefficient of $z^k$ term in the power series expansion of $\Phi_{N}(zq,zq,q/z,q/z)$.
This leads us to a new combinatorial theorem. 

\begin{theorem}\label{CombPHIExtraction} Let $N$, $n$, and $k$ be non-negative integers. Then \[\mathcal{A}_{N}(n,k) = \mathcal{B}_{N}(n,k),\] where $\mathcal{A}_{N}(n,k)$ is the number of partitions of $n$ into no more than $N$ parts, where exactly $k$ parts are odd, and $\mathcal{B}_{N}(n,k)$ is the number of partitions of $n$ into parts $\leq N$, where alternating sum of parts is equal to $k$.
\end{theorem}

We illustrate Theorem~\ref{CombPHIExtraction} in Table~\ref{TableExample_N_n_k}.

\begin{center}
\begin{table}[htb] \caption{$\mathcal{A}_{3}(10,2)$ and $\mathcal{B}_{3}(10,2)$ with respective partitions for Theorem~\ref{CombPHIExtraction}}\begin{tabular}{cc}
$\mathcal{A}_{3}(10,2)=9$ : & $\begin{array}{c}
(9,1),\ (8,1,1),\ (7,3),\ (7,2,1),\ (6,3,1),\  \vspace*{.1cm}\\
(5,5),\ (5,4,1),\ (5,3,2),\ (4,3,3).\vspace*{.1cm}
\end{array}$ \\ 
\\ [-1.5ex]
$\mathcal{B}_{3}(10,2)=9$ : & $\begin{array}{c}\vspace*{.1cm}(3,3,3,1),\ (3,3,2,1,1),\ (3,2,2,2,1),\\ \vspace*{.1cm} (3,2,2,1,1,1),\ (3,2,1,1,1,1,1),\ (3,1,1,1,1,1,1,1),\\ \vspace*{.1cm} (2,2,2,2,2),\ (2,2,2,1,1,1,1),\ (2,1,1,1,1,1,1,1,1). \vspace*{.1cm}\end{array}$ 
\end{tabular}\label{TableExample_N_n_k}  
\end{table}
\end{center}

It is clear that the partitions counted by $\mathcal{A}_{N}(n,k)$ and $\mathcal{B}_N(n,k)$ are conjugates of each other. This follows easily from the observation that the number of odd parts in a partition turns into the alternating sum of parts in the conjugate of this partition.

\section{Further Observations}\label{section6}
%%%%%%%%%%%%%%%%%%%%%%%%%%%%%%%%%%%%%%%%%%%%%%%%%%%%%%%%%%%%%%%%%%%%%%%%%%%%%%%%%%%%%%%%%%%%%%%%%%%%%

In this section we will extend Boulet's approach to the weighted partitions with bounds on the number of parts and largest parts., Theorem~\ref{boulet}.

Let $U_N$ be defined, as in Section~\ref{section3}, as the set of partitions with largest part less than or equal to $N$. Let $\wTilde{D}_N$ be the set of partitions into parts $\leq N$, where the difference between an odd-indexed part and the following non-zero even-indexed part is $\leq 1$. Let $\wTilde{E}_N$ be the set of partitions into parts $\leq N$, where the Ferrers diagrams of these partitions, exclusively, have odd-height columns, and every present column size repeats an even number of times. Define $\wTilde{D}$ and $\wTilde{E}$ similar to the sets $\wTilde{D}_N$ and $\wTilde{E}_N$, where we remove the restriction on the largest part. 

Let $\rho^*:U\mapsto \wTilde{D} \times \wTilde{E}$ ($\rho^*_N:U_N\mapsto \wTilde{D}_N \times \wTilde{E}_N$) be the similar map to $\rho$ ($\rho_N$). Let $\wTilde{\pi}$ be a partition in $U_N$. The image $\rho^*_N(\wTilde{\pi})=(\wTilde{\pi}',\wTilde{\pi}^*)$ is obtained by extracting even number of odd height columns from $\wTilde{\pi}$'s Ferrers diagram repeatedly, until there are no more repetitions of odd height columns in the Ferrers diagram of $\wTilde{\pi}$. We put these extracted columns in $\wTilde{\pi}^*$, and the partition that is left after extraction is $\wTilde{\pi}'$. An example of this map is $\rho^*((10,7,4)) = (\wTilde{\pi}',\wTilde{\pi}^*) = ((4,3),(6,4,4))$, as demonstrated in Table~\ref{table5}.

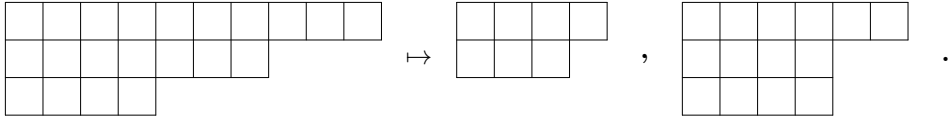
\begin{center} 
\definecolor{cqcqcq}{rgb}{0.75,0.75,0.75}
\begin{table}[htb]\caption{$\rho^*((10,7,4)) = ((4,3),(6,4,4))$}
\begin{tikzpicture}[line cap=round,line join=round,x=0.5cm,y=0.5cm]

\clip(0.5,0.5) rectangle (26.5,4.5);
\draw [line width=.25pt] (1,1)-- (5,1);
\draw [line width=.25pt] (1,2)-- (8,2);
\draw [line width=.25pt] (1,3)-- (11,3);
\draw [line width=.25pt] (1,4)-- (11,4);
\draw [line width=.25pt] (11,3)-- (11,4);
\draw [line width=.25pt] (10,4)-- (10,3);
\draw [line width=.25pt] (9,4)-- (9,3);
\draw [line width=.25pt] (8,2)-- (8,4);
\draw [line width=.25pt] (7,4)-- (7,2);
\draw [line width=.25pt] (6,2)-- (6,4);
\draw [line width=.25pt] (5,1)-- (5,4);
\draw [line width=.25pt] (1,4)-- (1,1);
\draw [line width=.25pt] (2,4)-- (2,1);
\draw [line width=.25pt] (3,4)-- (3,1);
\draw [line width=.25pt] (4,4)-- (4,1);
\draw [line width=.25pt] (13,4)-- (13,2);
\draw [line width=.25pt] (16,2)-- (16,4);
\draw [line width=.25pt] (17,4)-- (13,4);
\draw [line width=.25pt] (17,3)-- (13,3);
\draw [line width=.25pt] (17,4)-- (17,3);
\draw [line width=.25pt] (16,2)-- (13,2);
\draw [line width=.25pt] (14,4)-- (14,2);
\draw [line width=.25pt] (15,4)-- (15,2);
\draw [line width=.25pt] (19,4)-- (19,1);
\draw [line width=.25pt] (19,4)-- (25,4);
\draw [line width=.25pt] (24,4)-- (24,3);
\draw [line width=.25pt] (25,4)-- (25,3);
\draw [line width=.25pt] (25,3)-- (19,3);
\draw [line width=.25pt] (23,4)-- (23,1);
\draw [line width=.25pt] (19,1)-- (23,1);
\draw [line width=.25pt] (23,2)-- (19,2);
\draw [line width=.25pt] (20,4)-- (20,1);
\draw [line width=.25pt] (21,4)-- (21,1);
\draw [line width=.25pt] (22,4)-- (22,1);

\draw (12,2.5) node[anchor=center] {$\mapsto$};
\draw (18,2.5) node[anchor=center] {\textbf{,}};

\draw (26,2.5) node[anchor=center] {\textbf{.}};

\end{tikzpicture}\label{table5}
\end{table}
\end{center}

\noindent With the definition of the bijection $\rho^*_N$, we can finitize Boulet's combinatorial approach. Let $\wTilde{\pi}$ be a fixed partition with largest part less than or equal to $2N+\nu$ for a non-negative integer $N$ and $\nu\in\{0,1\}$. We look at $\rho_{2N+\nu}^*(\wTilde{\pi}) = (\wTilde{\pi}',\wTilde{\pi}^*)$. Here $\wTilde{\pi}'$ is a partition in $D_{2(N-k)+\nu}$ with the specified difference conditions. In this construction, $k$ is half the number of parts in $\wTilde{\pi}^*$ which is a partition in $E_{2k}$. The generating function for the weighted count of four-variable decorated Ferrers diagrams of such a partition $\wTilde{\pi}'$ is 
\begin{equation}\label{a_abc_ac_Q_parts_GF}\frac{(-a;Q)_{N-k+\nu} (-abc;Q)_{N-k}}{(ac;Q)_{N-k+\nu}(Q;Q)_{N-k}}.\end{equation} 
Similarly, the generating function for the weighted count of four-variable decorated Ferrers diagrams of such $\wTilde{\pi}^*$ is 
\begin{equation}\label{abPartsGF}\frac{(ab)^k}{(Q;Q)_k}.\end{equation} 
Hence, for $\nu\in\{0,1\}$, the generating function $\Phi_{2N+\nu}(a,b,c,d)$ for the weighted count of partitions with parts less than or equal to $2N+\nu$, is the sum over $k$ of the product of two functions in \eqref{a_abc_ac_Q_parts_GF} and \eqref{abPartsGF}. In this way, we arrive at
\begin{theorem}\label{finiteBoulet} For a non-zero integer $N$, variables $a$, $b$, $c$, $d$, and $Q=abcd$, we have
\begin{align}
\label{finiteBouletPHI}\Phi_{2N+\nu}(a,b,c,d) &=\frac{1}{(Q;Q)_{N}}\sum_{i=0}^N \genfrac{[}{]}{0pt}{}{N}{i}_{Q} \frac{(-a;Q)_{i+\nu}(-abc;Q)_{i} }{(ac;Q)_{i+\nu}}(ab)^{N-i},\\
\label{finiteBouletPSI}\Psi_{2N+\nu}(a,b,c,d) &=\sum_{i=0}^N \genfrac{[}{]}{0pt}{}{N}{i}_{Q} (-a;Q)_{i+\nu}(-abc;Q)_{i} \frac{(ac;Q)_{N+\nu}}{(ac;Q)_{i+\nu}}(ab)^{N-i},
\end{align} where $\nu\in\{0,1\}$.
\end{theorem}
\noindent We remark that we use \eqref{PSItoPHI_Bounded} to derive \eqref{finiteBouletPSI}. Observe that Theorem~\ref{finiteBoulet} is a perfect companion to Theorem~\ref{MasaoThm}. However, our derivation of Theorem~\ref{finiteBoulet}, unlike Theorem~\ref{MasaoThm}, is completely combinatorial. 

Next we rewrite \eqref{PSI2Nnu} and \eqref{finiteBouletPSI} using hypergeometric notations as
 \begin{align}\label{Psi_2Phi1}\Psi_{2N+\nu}(a,b,c,d) &= (ab)^N(-c;Q)_{N} (1+\nu a)\  {}_2\phi_1 \left(\genfrac{}{}{0pt}{}{Q^{-N},\ -aQ^{\nu}}{-\frac{Q^{1-N}}{c}};Q,-d \right),\\ \intertext{and} \nonumber
\Psi_{2N+\nu}(a,b,c,d)&= (-a^2b)^NQ^{N\nu} (1+\nu a) \frac{(-c,bdQ^{-N-\nu};Q)_N}{(-\frac{Q^{1-N}}{c};Q)_N}\times\\\label{3Phi1Sum}  &\hspace{2.7cm}{}_3\phi_1\left(\genfrac{}{}{0pt}{}{\ Q^{-N},\ -aQ^{\nu},\ -abc\ }{acQ^{\nu}}; Q, \frac{Q^N}{ab}\right).\end{align}
Comparing \eqref{Psi_2Phi1} and \eqref{3Phi1Sum} we get
\begin{equation}\label{transform2Phi1to3Phi1}
{}_2\phi_1 \left(\genfrac{}{}{0pt}{}{Q^{-N},\ -aQ^{\nu}}{-\frac{Q^{1-N}}{c}};Q,-d \right) = (-aQ^\nu)^N \frac{( bdQ^{-N-\nu};Q)_N}{(-\frac{Q^{1-N}}{c};Q)_N} {}_3\phi_1\left(\genfrac{}{}{0pt}{}{\ Q^{-N},\ -aQ^{\nu},\ -abc\ }{acQ^{\nu}}; Q, \frac{Q^N}{ab}\right).
\end{equation}

It is easy to check that \eqref{transform2Phi1to3Phi1} is nothing else but \cite[(III.8)]{GasperRahman} with the choice of variables $q\mapsto Q$, $b\mapsto-aQ^\nu$, $c\mapsto -Q^{1-N}/c$, and $z\mapsto -d$ for $\nu\in\{0,1\}$. 
%
%In \cite{AndrewsStanley}, Andrews notably finds the \eqref{finiteBouletPHI} representation of $\Phi_N(qzy,~qy/z,~qz/y,~q/zy)$. He sees that there is a connection between both representations using ${}_3\phi_2$ transformation \cite[(III.13)]{GasperRahman}. This empirical discovery is a direct consequence of \eqref{finiteBouletPHI} with the above shown variable choices.

We can further generalize \eqref{finiteBouletPHI} of Theorem~\ref{finiteBoulet} by putting bounds on the number of parts in the given partitions. Let $\Phi_{N,M}(a,b,c,d)$ be the generating function of partitions with weights, where every part is less than or equal to $N$, and the number of non-zero parts is less than or equal to $M$.

\begin{theorem}\label{boundedFiniteBoulet} Let $\nu\in\{0,1\}$ and $2N+\nu$, $M$ be positive integers,  then
\begin{align}\label{boundedFiniteBouletEQN}
\nonumber \Phi_{2N+\nu,2M}(a,b,c,d) &= \sum_{l=0}^N \genfrac{[}{]}{0pt}{}{N-l+M-1}{N-l}_{Q} (ab)^{N-l} \sum_{m_2=0}^l (abc)^{m_2} Q^{m_2 \choose 2} \genfrac{[}{]}{0pt}{}{l}{m_2}_{Q} \times \\ &\hspace{-1cm} \sum_{m_1=0}^{l+\nu} a^{m_1} Q^{m_1 \choose 2} \genfrac{[}{]}{0pt}{}{l+\nu}{m_1}_{Q}  \sum_{n=0}^{M-m_1-m_2} \genfrac{[}{]}{0pt}{}{M+l-n-m_1-m_2}{M-n-m_1-m_2}_{Q} \frac{(Q^{l+\nu};Q)_n}{(Q;Q)_n} (ac)^n.
\end{align}
\end{theorem}

\noindent Proof of Theorem~\ref{boundedFiniteBoulet} utilizes the same maps $\rho_N$ and $\rho_N^*$ as in the proof of Theorem~\ref{finiteBoulet}. To account for the bounds on the number of parts, the previously used generating functions are being replaced with appropriate $q$-binomial coefficients. Summing over the possibilities as we did in the proof of Theorem~\ref{finiteBoulet} yields Theorem~\ref{boundedFiniteBoulet}. For example, we replace product in \eqref{abPartsGF} with \[\genfrac{[}{]}{0pt}{}{M+k-1}{k}_{Q} (ab)^{k}.\] This is the generating function for the number of partitions into even parts such that the total number of parts is odd and $\leq 2M-1$, with the largest part being exactly $2k$, where we count these partitions with Boulet-Stanley weights.

\noindent Fixing $(a,b,c,d)=(q,q,q,q)$, we get \begin{equation}\label{qBinCoeff}\Phi_{N,2M}(q,q,q,q)=\genfrac{[}{]}{0pt}{}{N+2M}{2M}_{q}. \end{equation} 
So it makes sense to view the sum in \eqref{boundedFiniteBouletEQN} as a four-parameter generalization of the $q$-binomial coefficient \eqref{qBinCoeff}. 

We can combinatorially see that $\Phi_{2N+\nu,2M}(a,b,c,d)$ satisfies similar recurrence relations to $q$-binomial coefficients. Using these recurrences, we can write the $\Phi_{N,M}(a,b,c,d)$ for an odd $M$. We have the relations
\begin{equation}\label{restOfPhi}
\Phi_{2N+\nu,2M+1}(a,b,c,d) = \frac{\Phi_{2N+\nu,2(M+1)}(c,d,a,b)-\Phi_{2N-1+\nu,2(M+1)}(c,d,a,b)}{c^{\nu}(cd)^N},\\
\end{equation} where $N$, $M$ are non-negative integers and $\nu\in\{0,1\}$. This gives us the full list of possibilities for the bounds of $\Phi_{N,M}(a,b,c,d)$.

Yee in \cite{AeJaYee}, wrote a generating function for some weighted count of partitions with bounds on the largest part and the number of parts. We can also remove Yee's restrictions on weights. In other words, Yee's combinatorial study can be generalized to deal with the four-variable decorated Ferrers diagrams.

\begin{theorem}\label{generalYee} Let $2N+\nu$ and $2M+\mu$ be positive integers. Then 
\begin{align}\nonumber
\Phi_{2N+\nu,2M+\mu}(a,b,c,d) &= \sum_{k=0}^M (ac)^k \genfrac{[}{]}{0pt}{}{N+k-1+\nu}{k}_{Q} \sum_{j=0}^N (ab)^{N-j} \times\\ \label{generalYeeEQN}
&\hspace{.5cm}\left(\sum_{m_1=0}^j (1+\nu\mu a) a^{m_1} Q^{{m_1\choose 2}+\nu\mu m_1} \genfrac{[}{]}{0pt}{}{M-k+\mu-\nu}
{m_1}_{Q}\genfrac{[}{]}{0pt}{}{M-k+j-m_1}{j-m_1}_{Q} \right)\times\\ \nonumber
&\hspace{.5cm}\left(\sum_{m_2=0}^{N-j} c^{m_2} Q^{m_2\choose 2} \genfrac{[}{]}{0pt}{}{M-k}{m_2}_{Q} \frac{(Q^{M-k+\mu};Q)_{N-j-m_2}}{(Q;Q)_{N-j-m_2}} \right)
\end{align} for $\nu,\ \mu\in\{0,1\}$, where $(\nu,\mu)\not=(1,0)$ and $Q=abcd$.
\end{theorem} 

Theorem~\ref{boundedFiniteBoulet} and Theorem~\ref{generalYee} give different expressions for the same function. Also note that the idea behind \eqref{restOfPhi} can be applied to Theorem~\ref{generalYee} for the missing combination of bounds. This is left as an exercise for the interested reader.

Another important result comes from knowing the combinatorial generating function interpretations of the function $\Psi_N(a,b,c,d)$. By looking at different choices of the variables, we can find $q$-series identities that were not combinatorially interpreted before. One finding of this type is a $q$-hypergeometric identity of Berkovich and Warnaar, \cite[(3.30)]{WarnaarBerkovich}, and it's analogue.

For $\nu\in\{0,1\}$, it is obvious that $\Psi_{2N+\nu}(aq,q/a,aq,q/a)$ is the generating function of partitions into distinct parts less than or equal to $2N+\nu$, where the exponent of $a$ counts the number of odd parts in the partitions. Clearly, \begin{equation}\label{PsiDistinctProduct}\Psi_{2N+\nu}(aq,q/a,aq,q/a) = (-aq;q^2)_{N+\nu}(-q^2;q^2)_N .\end{equation} 
Using \eqref{PSI2Nnu} yields \begin{equation}\label{BerkoWarnaarPSI}\Psi_{2N+\nu}(aq,q/a,aq,q/a) = (-aq;q^4)_{N+\nu}\ {}_2\phi_1 \left(\genfrac{}{}{0pt}{}{q^{-4N},\ -aq}{-(aq)^{-1}q^{-4N-\nu}};q^4,-q^{1+4\delta_{\nu,0}}/a \right).\end{equation} 
Comparing \eqref{PsiDistinctProduct} and \eqref{BerkoWarnaarPSI} gives

\begin{theorem}\label{BerkovichWarnaar} For a non-negative integer $N$ and variables $a$ and $q$ \[ {}_2\phi_1 \left(\genfrac{}{}{0pt}{}{q^{-4N},\ -aq}{-(aq)^{-1}q^{-4(N-\delta_{\nu,0})}};q^4,-q^{1+4\delta_{\nu,0}}/a \right) = \frac{(-q^2;q^2)_N(-aq^3;q^2)_{N-1+\nu}}{(-aq^5;q^4)_{N-1+\nu}},\]where $\nu\in\{0,1\}$ and $\delta_{i,j}$ is the \textit{Kronecker delta function}.\end{theorem} 

The case $\nu=1$, with $N\mapsto n$, $a\mapsto -a/q$ and $q^2\mapsto q$ in Theorem~\ref{BerkovichWarnaar}, is \cite[(3.30)]{WarnaarBerkovich} and $\nu=0$ is the easy analogue we get using \eqref{PsiDistinctProduct}.

\section{Outlook}
%%%%%%%%%%%%%%%%%%%%%%%%%%%%%%%%%%%%%%%%%%%%%%%%%%%%%%%%%%%%%%%%%%%%%%%%%%%%%%%%%%%%%%%%%%%%%%%%%%%%%

Theorems \ref{BasicCorollaryComb}, \ref{distinctBG} and \ref{unrestrictedBG} call for combinatorial proofs, which we will discuss elsewhere.

We have found the doubly bounded versions of Theorem~\ref{distinctBG} and Theorem~\ref{combHermiteSum}. Let $N$ and $M$ be non-negative integers. Note that $\Phi_{N,M}(qt,q/t,q/t,qt)$ is the generating function for the number of partitions into parts $\leq N$ and number of parts $\leq M$, where the exponent of $t$ counts the BG-rank of counted partitions. Recall that $\Phi_{N,M}(zq,zq,q/z,q/z)$ is the generating function for the number of partitions into parts $\leq N$ and number of parts $\leq M$, where the exponent of $z$ is the alternating sum of parts of counted partitions. 

\begin{proposition}\label{BG_Double_Bounded} Let $\nu,\ \mu \in\{0,1\}$ and $2N+\nu$, $2M+\mu$ be positive integers, then \begin{equation}\label{BG_double_bdd_eqn}\Phi_{2N+\nu,2M+\mu}(qt,q/t,q/t,qt) = \sum_{j=-N}^{N+\nu} t^j q^{2j^2-j} \genfrac{[}{]}{0pt}{}{N+M+\mu}{N+j}_{q^2} \genfrac{[}{]}{0pt}{}{N+\nu+M}{M+j}_{q^2}.\end{equation}
\end{proposition}

We remark that the products of $q$-binomial coefficients that are similar to those on the right-hand side of \eqref{BG_double_bdd_eqn} were studied by Burge in \cite{Burge}.

\begin{proposition}\label{AS_Double_Bounded} Let $\mu \in\{0,1\}$ and $N$, $2M+\mu$ be positive integers, then \[\Phi_{N,2M+\mu}(zq,zq,q/z,q/z) = \sum_{j=0}^{N} z^j q^{j} \genfrac{[}{]}{0pt}{}{M+\mu+j-1}{j}_{q^2} \genfrac{[}{]}{0pt}{}{M+N-j}{M}_{q^2}.\]
\end{proposition}

\noindent Both of these results, compared with either \eqref{boundedFiniteBouletEQN} or \eqref{generalYeeEQN}, with their appropriate choices of variables, give us new $q$-series identities relating single-fold sums and four-fold sums. These identities will be discussed elsewhere.

It is easy to check that $M\rightarrow\infty$ in the Propositions \ref{BG_Double_Bounded} and \ref{AS_Double_Bounded}, implies \eqref{unrestrictedBGtermWithoutT} and \eqref{extractionPHI}, respectively. Proposition~\ref{BG_Double_Bounded} is symmetric with respect to its bounds on the number of parts and largest parts, which is expected because the BG-rank of a partition is invariant under conjugation. Hence, as $N\rightarrow\infty$ we see that Proposition~\ref{BG_Double_Bounded} implies \eqref{unrestrictedBGtermWithoutT}, as in the case of $M\rightarrow\infty$. 

In Proposition~\ref{AS_Double_Bounded}, by extracting the coefficient of $z^k$, we get a new partition identity.

\begin{proposition}\label{DoubleBDD_AS_Combinat_result} Let $\mu\in\{0,1\}$, $N$ and $2M+\mu$ positive integers, and $n$ and $k$ be non-negative integers. Then, \[\mathcal{A}_{2M+\mu,N}(n,k) = \mathcal{B}_{N,2M+\mu}(n,k), \]
where $\mathcal{A}_{2M+\mu,N}(n,k) $ is the number of partitions into parts $\leq 2M+\mu$, with no more than $N$ parts where exactly $k$ parts are odd, and $\mathcal{B}_{N,2M+\mu}(n,k) $ is the number of partitions of $n$ into parts $\leq N$, with no more than $2M+\mu$ parts with alternating sum of parts equal to $k$.
\end{proposition}
We illustrate this result in Table~\ref{LastExampleTable}

\begin{center}
\begin{table}[htb] \caption{$\mathcal{A}_{5,3}(10,2)$ and $\mathcal{B}_{3,5}(10,2)$ with respective partitions for Proposition~\ref{DoubleBDD_AS_Combinat_result}}\begin{tabular}{cc}
$\mathcal{A}_{5,3}(10,2)=4$ : & $\begin{array}{c}
(5,5),\ (5,4,1),\ (5,3,2),\ (4,3,3).\vspace*{.1cm}
\end{array}$ \\ 
\\ [-1.5ex]
$\mathcal{B}_{3,5}(10,2)=4$ : & $\begin{array}{c}\vspace*{.1cm}(3,3,3,1),\ (3,3,2,1,1),\ (3,2,2,2,1),\ (2,2,2,2,2).\end{array}$ 
\end{tabular}\label{LastExampleTable}  
\end{table}
\end{center}
\noindent
The set of partitions that appears in Table~\ref{LastExampleTable}, is a subset of the ones in Table~\ref{TableExample_N_n_k}. 

Proposition~\ref{DoubleBDD_AS_Combinat_result} is a generalized version of Theorem~\ref{CombPHIExtraction} with the extra bound on the number of parts. There is an obvious connection between notations of Theorem~\ref{CombPHIExtraction} and Proposition~\ref{DoubleBDD_AS_Combinat_result}, $\mathcal{A}_{M,N}(n,k)\rightarrow \mathcal{A}_{N}(n,k)$ ($\mathcal{B}_{N,M}(n,k)\rightarrow \mathcal{B}_{N}(n,k)$), as $M\rightarrow\infty$. Similar to the case of Theorem~\ref{CombPHIExtraction}, the partitions counted by $\mathcal{A}_{N,2M+\mu}(n,k)$ and $\mathcal{B}_{2M+\mu,N}(n,k) $ are conjugates of each other. 

The approach used in Section~\ref{section2} can be generalized to apply for more general generating functions. Let $\wTilde{P}_{N}(i,j,m,q)$ be the generating function for the number of partitions into distinct parts $\leq N$ with $i$ odd-indexed, $j$ even-indexed odd parts, and $m$ even parts. Obviously we have \[\sum_{m\geq 0} \wTilde{P}_{N}(i,j,m,q) = P_{N}(i,j,q),\] where $P_{N}(i,j,q)$ was defined in Section~\ref{section2}.

Authors intend to discuss these more general generating functions elsewhere. To this end, we already proved that
\begin{proposition}\label{StrongProp} Let $i$, $j$, $m$ and $N$ be non-negative integers. Then,
\begin{align*}
\wTilde{P}_{2N+\nu}(0,j,m,q) &= q^{j(j+1)+m(m+1)-j(-1)^{m+j}} \genfrac{[}{]}{0pt}{}{N+j}{j+m}_{q^2} \genfrac{[}{]}{0pt}{}{\lfloor \frac{m+j}{2} \rfloor }{j}_{q^4},\\
\wTilde{P}_{2N+1}(i,0,m,q) &= q^{i(i+1)+m(m+1)+i(-1)^{m+i}} \genfrac{[}{]}{0pt}{}{N+i}{i+m}_{q^2} \genfrac{[}{]}{0pt}{}{\lceil \frac{m+i}{2} \rceil }{i}_{q^4},\\
\wTilde{P}_{2N}(i,0,m,q)  &=q^{i(i+1)+m(m+1)+i(-1)^{m+i}} \genfrac{[}{]}{0pt}{}{N+i-1}{i+m}_{q^2} \genfrac{[}{]}{0pt}{}{\lceil \frac{m+i}{2} \rceil }{i}_{q^4} \\  &\hspace{1cm}+q^{i(i+1)+m(m-1) +i(-1)^{m+i}+2N} \genfrac{[}{]}{0pt}{}{N+i-1}{i+m-1}_{q^2} \genfrac{[}{]}{0pt}{}{\lfloor \frac{m+i-1}{2} \rfloor }{i}_{q^4}, \end{align*}
where $\nu\in\{0,1\}$.
\end{proposition}

\noindent
We illustrate Proposition~\ref{StrongProp} in Table~\ref{Table8}

\begin{table}[htb] \caption{Various choices of $(N,i,j,m)$ in Proposition~\ref{StrongProp}}\begin{tabular}{lcc}
$\wTilde{P}_{N}(i,j,m,q)$   \hspace{1cm} Expression & & Relevant partitions.\\ [-1.2ex] \\
${\begin{aligned}
\wTilde{P}_{7}(0,1,2,q)&=q^9\genfrac{[}{]}{0pt}{}{4}{3}_{q^2}\genfrac{[}{]}{0pt}{}{1}{1}_{q^4}\\[-1.7ex]\\  &=q^9+q^{11}+q^{13}+q^{15}\end{aligned}}$ &: & $(4,3,2),\ (6,3,2),\ (6,5,2),\ (6,5,4).$ \\ [-1.5ex]\\
${\begin{aligned}\wTilde{P}_{7}(1,0,1,q) &= q^5\genfrac{[}{]}{0pt}{}{4}{2}_{q^2}\genfrac{[}{]}{0pt}{}{1}{1}_{q^4}\\[-1.7ex]\\  &=q^5+q^7+2q^9+q^{11}+q^{13}\end{aligned}}$ &: & $(3,2),\ (5,2),\ (5,4),\ (7,2),\ (7,4),\ (7,6).$ \\[-1.5ex]\\ 
$\begin{aligned}\wTilde{P}_{6}(2,0,1,q)&=q^6\genfrac{[}{]}{0pt}{}{4}{3}_{q^2}\genfrac{[}{]}{0pt}{}{2}{2}_{q^4}+q^{10}\genfrac{[}{]}{0pt}{}{4}{2}_{q^2}\genfrac{[}{]}{0pt}{}{1}{2}_{q^4}\\[-1.7ex]\\  &=q^6+q^8+q^{10}+q^{12}\end{aligned}$ &: & $(3,2,1),\ (5,2,1),\ (5,4,1),\ (5,4,3).$ \\ 
\end{tabular}\label{Table8}  
\end{table}

We comment that Proposition~\ref{StrongProp} provides significant refinement of the Savage-Sills generating functions in \cite{SavageSills}.

\section{Acknowledgment} 
%%%%%%%%%%%%%%%%%%%%%%%%%%%%%%%%%%%%%%%%%%%%%%%%%%%%%%%%%%%%%%%%%%%%%%%%%%%%%%%%%%%%%%%%%%%%%%%%%%%%%

Authors would like to thank Krishnaswami Alladi, George Andrews, Johann Cigler, Sinai Robins, and Li-Chien Shen for their kind interest and helpful comments. Special thanks to Frank Patane for careful reading of the manuscript and thoughtful suggestions.

\end{document}